\documentclass[a4paper,11pt]{article}

\usepackage{amsmath}
\usepackage{amssymb}
\usepackage{amsthm}
\usepackage{graphicx}
\usepackage{amscd}
\usepackage{epic, eepic}
\usepackage{url}
\usepackage{color}
\usepackage[utf8]{inputenc} 
\usepackage{comment}
\usepackage{enumerate}   
\usepackage{todonotes}
\usepackage{graphicx}
\usepackage{epstopdf}
\usepackage{enumitem}
\usepackage{tikz}
\usepackage[top=22mm, bottom=23mm, left=22mm, right=22mm]{geometry}
\usepackage[bookmarks=false,hidelinks]{hyperref}
\usepackage{mathtools}

\makeatletter
\renewenvironment{proof}[1][\proofname] {\par\pushQED{\qed}\normalfont\topsep6\p@\@plus6\p@\relax\trivlist\item[\hskip\labelsep\bfseries#1\@addpunct{.}]\ignorespaces}{\popQED\endtrivlist\@endpefalse}
\makeatother

\newtheorem{theorem}{\bf Theorem}[section]
\newtheorem{lemma}[theorem]{\bf Lemma}
\newtheorem{corollary}[theorem]{\bf Corollary}
\newtheorem{proposition}[theorem]{\bf Proposition}
\newtheorem{conjecture}[theorem]{\bf Conjecture}
\newtheorem{question}[theorem]{\bf Question}

\theoremstyle{definition}

\newtheorem{remark}[theorem]{\bf Remark}
\newtheorem{definition}[theorem]{\bf Definition}

\newtheorem{example}[theorem]{\bf Example}

\def\eps{\varepsilon}
\def\vx{\textbf{x}}

\def\cH{\mathcal{H}}
\def\ex{\mathrm{ex}}

\title{\vspace{-0.9cm}On the Turán number of the hypercube}
\author{Oliver Janzer\thanks{Department of Pure Mathematics and Mathematical Statistics, University of Cambridge, United Kingdom. Research supported by a research fellowship at Trinity College. Email: \textbf{oj224@cam.ac.uk}.}
	\and
	Benny Sudakov\thanks{Department of Mathematics, ETH Z\"urich, Switzerland. Research supported in part by SNSF grant 200021\_196965. Email: \textbf{benjamin.sudakov@math.ethz.ch}.}}
\date{}

\begin{document}

\maketitle

\begin{abstract}
    In 1964, Erd\H os proposed the problem of estimating the Tur\'an number of the $d$-dimensional hypercube $Q_d$. Since $Q_d$ is a bipartite graph with maximum degree $d$, it follows from results of F\"uredi and Alon, Krivelevich, Sudakov that $\ex(n,Q_d)=O_d(n^{2-1/d})$. A recent general result of Sudakov and Tomon implies the slightly stronger bound $\ex(n,Q_d)=o(n^{2-1/d})$. We obtain the first power-improvement for this old problem by showing that $\ex(n,Q_d)=O_d(n^{2-\frac{1}{d-1}+\frac{1}{(d-1)2^{d-1}}})$. This answers a question of Liu. Moreover, our techniques give a power improvement
    for a larger class of graphs than cubes.
    
    We use a similar method to prove that any $n$-vertex, properly edge-coloured graph without a rainbow cycle has at most $O(n(\log n)^2)$ edges, improving the previous best bound of $n(\log n)^{2+o(1)}$ by Tomon. Furthermore, we show that any properly edge-coloured $n$-vertex graph with $\omega(n\log n)$ edges contains a cycle which is almost rainbow: that is, almost all edges in it have a unique colour. This latter result is tight.
\end{abstract}

\section{Introduction}

For a graph $H$ and positive integer $n$, the Tur\'an number (or extremal number) $\ex(n,H)$ is the maximum possible number of edges in an $n$-vertex graph which does not contain $H$ as a subgraph. By a result of Tur\'an \cite{Tur41}, the exact value of this function is known when $H$ is a complete graph. More generally, the function is well-understood for graphs with chromatic number at least 3 by the celebrated Erd\H os--Stone--Simonovits theorem \cite{ES46,ES65} which states that $$\ex(n,H)=\left(1-\frac{1}{\chi(H)-1}+o(1)\right)\binom{n}{2}.$$
However, for bipartite graphs $H$, the known bounds are much less accurate. It is known that for any bipartite graph $H$ there is some $\eps=\eps(H)>0$ such that $\ex(n,H)=O(n^{2-\eps})$ and it is conjectured that in fact there is some $\alpha=\alpha(H)$ such that $\ex(n,H)=\Theta(n^{\alpha})$. However, this is not known even for some very simple graphs such as the complete bipartite graph $K_{4,4}$, the even cycle $C_8$ and the 3-dimensional cube $Q_3$. In 1964, Erd\H os \cite{Erd64} wrote that Tur\'an had proposed the study of the extremal number of the five platonic solids (to be more precise, that of the graph of these polyhedra). The graph of the tetrahedron is $K_4$, so its extremal number is known by Tur\'an's theorem. Erd\H os and Simonovits determined the Tur\'an number of the octahedron \cite{ES71}, and Simonovits determined the extremal number of the dodecahedron \cite{Sim74dodecahedron} and the icosahedron \cite{Sim74}. However, the case of the cube is much more difficult as, unlike the other solids, its graph is bipartite.

In the same paper from 1964, Erd\H os \cite{Erd64} also mentions the problem of determining the Tur\'an number of higher-dimensional cubes. The $d$-dimensional cube $Q_d$ is the graph whose vertex set is $\{0,1\}^d$ and in which two vertices are joined by an edge if they differ in exactly one coordinate.
In 1969, Erd\H os and Simonovits \cite{ES69} proved that $\ex(n,Q_3)=O(n^{8/5})$ which is still the the best known upper bound for this problem. The best known lower bound is $\ex(n,Q_3)=\Omega(n^{3/2})$ and follows from the observation that $Q_3$ contains a $4$-cycle. Any improvement on these long-standing bounds would be considered a major breakthrough.

The high-dimensional case seems to be even more challenging. It can be derived from a result of F\"uredi \cite{Fur91} that if $H$ is a bipartite graph with maximum degree at most $d$ on one side, then $\ex(n,H)=O(n^{2-1/d})$. Alon, Krivelevich and Sudakov \cite{AKS03} gave another proof of this estimate as one of the first applications of the celebrated dependent random choice method. Clearly, this implies in particular that $\ex(n,Q_d)=O(n^{2-1/d})$. A few years ago, Conlon and Lee made the following beautiful conjecture.

\begin{conjecture}[Conlon--Lee \cite{CL21}] \label{conj:conlonlee}
    Let $H$ be a $K_{d,d}$-free bipartite graph with maximum degree at most $d$ on one side. Then $$\ex(n,H)=O(n^{2-1/d-\eps})$$ holds for some $\eps=\eps(H)>0$.
\end{conjecture}

While Conjecture \ref{conj:conlonlee} is wide open, there are a few partial results towards it.
Conlon and Lee \cite{CL21} proved the conjecture in the special case $d=2$. Conlon, Janzer and Lee \cite{CJL21} showed that if $H$ is a $K_{2,2}$-free bipartite graph with maximum degree $d$ on one side, then $\ex(n,H)=o(n^{2-1/d})$. This was improved by Sudakov and Tomon who proved the following.

\begin{theorem}[Sudakov--Tomon \cite{ST22}] \label{thm:o bound}
    Let $H$ be a $K_{d,d}$-free bipartite graph with maximum degree at most $d$ on one side. Then $$\ex(n,H)=o(n^{2-1/d}).$$
\end{theorem}

Since for $d\geq 3$, $Q_d$ does not contain $K_{d,d}$ as a subgraph, Theorem \ref{thm:o bound} implies that $\ex(n,Q_d)=o(n^{2-1/d})$. Liu asked the following question.

\begin{question}[Liu \cite{Liu21}]
    Let $d\geq 3$ be an integer. Is it true that there exists some $\eps=\eps(d)>0$ such that
    $$\ex(n,Q_d)=O(n^{2-1/d-\eps})?$$
\end{question}

We answer this question affirmatively by proving the first power-improvement over the dependent random choice bound.

\begin{theorem} \label{thm:turan cube}
    For any integer $d\geq 3$, $$\ex(n,Q_d)=O_d\left(n^{2-\frac{1}{d-1}+\frac{1}{(d-1)2^{d-1}}}\right).$$
\end{theorem}

As a side note, we remark that an improvement for the Ramsey number of the hypercube was obtained very recently by Tikhomirov \cite{Tik22}. He showed that there is a positive constant $c$ such that $r(Q_n)=O(2^{2n-cn})$. This improved the previous best bound, $r(Q_n)=O(2^{2n})$, proved by Conlon, Fox and Sudakov \cite{CFS16} which had been established using the dependent random choice method. In fact, in both of these results, the proofs show that the denser of the two colours contains $Q_n$. Our result can be viewed as an analogue of Tikhomirov's result for the related Tur\'an problem (the difference being that our forbidden hypercube has constant size, but the host graph is much sparser). However, we point out that our methods are completely different from Tikhomirov's.

In 1984, about 15 years after their proof of the bound $\ex(n,Q_3)=O(n^{8/5})$, Erd\H os and Simonovits \cite{ES84} showed that in fact any $n$-vertex graph with more than $Cn^{8/5}$ edges has not just one, but at least as many copies (up to a constant factor) of $Q_3$ as a random graph with the same edge density. This phenomenon is called supersaturation.
We are able to get an analogous result for higher dimensions. We note that the previous proofs using dependent random choice or Theorem \ref{thm:o bound} did not give a supersaturation result even at those higher densities.

\begin{theorem} \label{thm:cube supersaturation}
    For any integer $d\geq 3$, there are positive constants $c=c(d)$ and $C=C(d)$ such that any $n$-vertex graph with edge density $p\geq Cn^{-\frac{1}{d-1}+\frac{1}{(d-1)2^{d-1}}}$ has at least $cn^{2^d}p^{d2^{d-1}}$ copies of $Q_d$.
\end{theorem}

\noindent Here and below we say that an $n$-vertex graph $G$ has edge density $p$ if it has $pn^2/2$ edges.

Our methods can also be applied to prove Conjecture \ref{conj:conlonlee} for a larger class of graphs. We will discuss the precise description of all graphs for which the technique is applicable in the next section. For now, we just highlight another family of graphs (known as the bipartite Kneser graphs) for which we can verify Conjecture \ref{conj:conlonlee}.

\begin{definition} \label{def:set graph}
    For $1\leq \ell<k/2$, the bipartite Kneser graph $H_{\ell,k}$ is the bipartite graph whose parts are $[k]^{(\ell)}$ and $[k]^{(k-\ell)}$ and in which $S\in [k]^{(\ell)}$ and $T\in [k]^{(k-\ell)}$ are joined by an edge if $S\subset T$. Note that $H_{\ell,k}$ is a regular graph.
\end{definition}

In the above definition and in what follows, $[k]^{(\ell)}$ stands for the family of subsets of size $\ell$ in $[k]$.

\begin{theorem} \label{thm:set graph}
    Let $d$ be the degree of the vertices in $H_{\ell,k}$. Then there is some $\eps=\eps(\ell,k)>0$ such that $\ex(n,H_{\ell,k})=O(n^{2-1/d-\eps})$.
\end{theorem}

We remark that with the same argument we could also prove a supersaturation result for $H_{\ell,k}$.

\subsection{Rainbow cycles}

We will also use our methods to improve the best known upper bound for finding rainbow cycles. The study of rainbow Tur\'an problems was initiated by Keevash, Mubayi, Sudakov and Verstra\"ete \cite{KMSV07}. They asked how many edges one can have in a properly edge-coloured $n$-vertex graph without containing a rainbow cycle (i.e., a cycle in which all edges have a different colour). Let us write $f(n)$ for this number. They observed that if the edges of a hypercube are coloured according to the ``direction" of the edge, then the resulting properly edge-coloured graph does not have a rainbow cycle (and in fact every colour that appears in a given cycle must appear at least twice in it). Hence, $f(n)=\Omega(n\log n)$. The first non-trivial upper bound was obtained by Das, Lee and Sudakov \cite{DLS13}, who showed that for any $\gamma>0$ and sufficiently large $n$, we have $f(n)\leq n\exp((\log n)^{1/2+\gamma})$. Janzer \cite{Jan20} proved that $f(n)=O(n(\log n)^4)$. The current best bound is due to Tomon \cite {Tom22} who showed that $f(n)\leq n(\log n)^{2+o(1)}$. We improve this further as follows.

\begin{theorem} \label{thm:rainbow cycle}
    If $n$ is sufficiently large, then any properly edge-colored $n$-vertex graph with at least $8n(\log n)^2$ edges contains a rainbow cycle.
\end{theorem}

Keevash, Mubayi, Sudakov and Verstra\"ete \cite{KMSV07} also proved that if $G$ is a properly edge-coloured $n$-vertex graph with at least $n\log_2(n+3)-2n$ edges, then for some $k$ it contains a cycle of length $k$ which has more than $k/2$ different colours. Because of the hypercube construction, this is tight up to a constant factor. We significantly strengthen this result by finding a cycle which is almost rainbow.

\begin{theorem} \label{thm:almost rainbow cycle}
    If $n$ is sufficiently large, $0<\eps<1/2$ and $G$ is a properly edge-coloured $n$-vertex graph with at least $\frac{4}{\eps}n\log n$ edges, then for some $k$ it contains a cycle of length $k$ with more than $(1-\eps)k$ different colours.
\end{theorem}

The rest of this paper is organized as follows. In the next section, we prove our results on ordinary Tur\'an numbers and supersaturation. In Section \ref{sec:rainbow cycles} we prove our results on rainbow and almost rainbow cycles. We finish the paper with some concluding remarks in Section \ref{sec:concluding remarks}.

\section{Ordinary Tur\'an numbers}

\subsection{Illustration of our method and some preliminaries}

In this subsection, we illustrate our method on the example of the 3-dimensional cube and prove the following result (which is of course slightly weaker than the result of Erd\H os and Simonovits~\cite{ES84} that obtains the same conclusion for graphs with edge density $p\geq Cn^{-2/5}$).

\begin{proposition} \label{prop:cube}
    There are positive constants $c$ and $C$ such that any $n$-vertex graph with edge density $p\geq Cn^{-3/8}$ contains at least $cn^8p^{12}$ copies of $Q_3$.
\end{proposition}

Given graphs $H$ and $G$, a \emph{homomorphism} from $H$ to $G$ is a map $V(H)\rightarrow V(G)$ which sends edges to edges. Often we call such a map a \emph{homomorphic copy} of $H$ in $G$. We write $\hom(H,G)$ for the number of homomorphisms from $H$ to $G$. 

The proof of Proposition \ref{prop:cube} is via an inequality between the number of certain homomorphic copies of $Q_3$ in $G$. More precisely, we show that if a positive proportion of the homomorphic copies of $Q_3$ in $G$ are not injective, then a positive proportion of the homomorphisms are actually very far from being injective: namely all four vertices in one part of the bipartition of $Q_3$ are mapped to the same vertex. However, the latter is the same as a homomorphic copy of a star with four edges in $G$, and we can easily bound the number of such copies from above by $n\Delta(G)^4$. Hence, as long as the number of homomorphic copies of $Q_3$ in $G$ is much bigger than $n\Delta(G)^4$, it follows that most homomorphisms from $Q_3$ to $G$ are injective (i.e., genuine labelled copies of $Q_3$). It is well-known that $Q_3$ satisfies Sidorenko's conjecture, therefore if $G$ has edge density $p$, then it contains $\Omega(n^8p^{12})$ homomorphic copies of $Q_3$. Now if $G$ has maximum degree $O(pn)$ (which can be assumed by standard reduction results), then we require $n^8p^{12}\gg n(pn)^4$, which is $p\gg n^{-3/8}$. This means that an $n$-vertex graph with edge density $\gg n^{-3/8}$ contains the desired number of copies of $Q_3$.

Let us prove the promised inequalities between the number of various homomorphisms $Q_3\rightarrow G$. For graphs $H$, $G$ and a set $R\subset V(H)$, let us write $\hom(H,G;R)$ for the number of graph homomorphisms $V(H)\rightarrow V(G)$ with the property that all vertices in $R$ are mapped to the same vertex in $G$.
Identify $V(Q_3)$ with $\{0,1\}^3=\{000,001,\dots,111\}$ (and see Figure \ref{fig:cube}). The key inequalities are as follows.

\begin{lemma} \label{lem:cube ineqs}
    For any graph $G$, we have $$\hom(Q_3,G;\{000,011\})^2\leq \hom(Q_3,G;\{000,011,101\})\hom(Q_3,G).$$
    Furthermore,
    $$\hom(Q_3,G;\{000,011,101\})^2\leq \hom(Q_3,G;\{000,011,101,110\})\hom(Q_3,G).$$
\end{lemma}

\begin{figure}
	\centering
	\begin{tikzpicture}[scale=0.8]
		\draw[fill=black](0,0)circle(3pt);
		\draw[fill=black](3,0)circle(3pt);
		\draw[fill=black](5,-5)circle(3pt);
		\draw[fill=black](8,-5)circle(3pt);
		\draw[fill=black](10,0)circle(3pt);
		\draw[fill=black](13,0)circle(3pt);
		\draw[fill=black](5,5)circle(3pt);
		\draw[fill=black](8,5)circle(3pt);

		\draw[thick](0,0)--(3,0)(5,-5)--(8,-5)(10,0)--(13,0)(5,5)--(8,5);
		\draw[thick](0,0)--(5,-5)(0,0)--(5,5)(10,0)--(5,-5)(10,0)--(5,5);
        \draw[thick](3,0)--(8,-5)(3,0)--(8,5)(13,0)--(8,-5)(13,0)--(8,5);
		
		\node at (0,-0.8)  {000};
		\node at (3,-0.8)  {001};
		\node at (5,-5.8)  {010};
		\node at (8,-5.8)  {011};
		\node at (10,-0.8)  {110};
		\node at (13,-0.8)  {111};
		\node at (5,4.2)  {100};
		\node at (8,4.2)  {101};
	\end{tikzpicture}
	\caption{The cube} \label{fig:cube}
\end{figure}
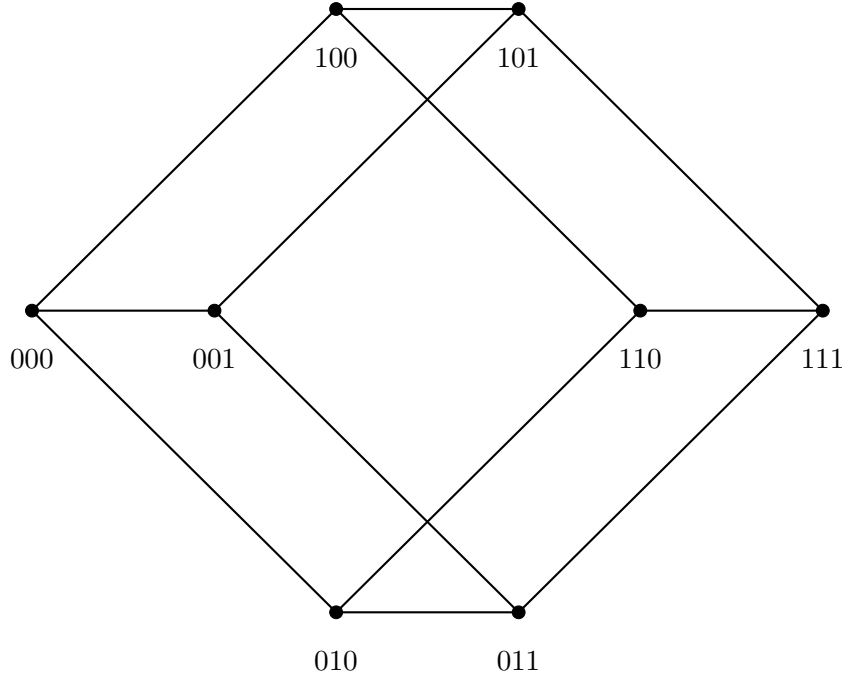

\begin{proof}
Let us start with the first inequality.
Let $f:Q_3[\{000,001,110,111\}]\rightarrow G$ be a homomorphism. Let $\alpha_f$ be the number of maps $g:\{010,011\}\rightarrow V(G)$ such that $f$ and $g$ together induce a homomorphism from $Q_3[\{000,001,110,111,010,011\}]$ to $G$. Note that by the symmetry of $Q_3$ this is the same as the number of maps $h:\{100,101\}\rightarrow V(G)$ such that $f$ and $h$ together induce a homomorphism from $Q_3[\{000,001,110,111,100,101\}]$ to $G$.

Let $\beta_f$ be the number of maps $g:\{010,011\}\rightarrow V(G)$ such that $f$ and $g$ together induce a homomorphism from $Q_3[\{000,001,110,111,010,011\}]$ to $G$, and in addition $g(011)=f(000)$. Note that by the symmetry of $Q_3$ this is the same as the number of maps $h:\{100,101\}\rightarrow V(G)$ such that $f$ and $h$ together induce a homomorphism from $Q_3[\{000,001,110,111,100,101\}]$ to $G$, and in addition $h(101)=f(000)$.

Now note that
$$\hom(Q_3,G;\{000,011\})=\sum_{f} \alpha_f \beta_f,$$
where the summation is over all homomorphisms $f:Q_3[\{000,001,110,111\}]\rightarrow G$. Indeed, $\alpha_f \beta_f$ is the number of suitable homomorphisms $\theta$ extending $f$ since there are $\alpha_f$ ways to choose $\theta|_{\{100,101\}}$, there are $\beta_f$ ways to choose $\theta|_{\{010,011\}}$, and any such pair is suitable because there are no edges between $\{100,101\}$ and $\{010,011\}$. Similarly,
$$\hom(Q_3,G;\{000,011,101\})=\sum_{f} \beta_f^2$$ and
$$\hom(Q_3,G)= \sum_{f} \alpha_f^2.$$
The required inequality follows from the Cauchy-Schwarz inequality.

Let us now prove the second inequality.
Let $f:Q_3[\{010,011,100,101\}]\rightarrow G$ be a homomorphism such that $f(101)=f(011)$. Let $\alpha_f$ be the number of maps $g:\{000,001\}\rightarrow V(G)$ such that $f$ and $g$ together induce a homomorphism from $Q_3[\{010,011,100,101,000,001\}]$ to $G$. Note that by the symmetry of $Q_3$ this is the same as the number of maps $h:\{110,111\}\rightarrow V(G)$ such that $f$ and $h$ together induce a homomorphism from $Q_3[\{010,011,100,101,110,111\}]$ to $G$.

Let $\beta_f$ be the number of maps $g:\{000,001\}\rightarrow V(G)$ such that $f$ and $g$ together induce a homomorphism from $Q_3[\{010,011,100,101,000,001\}]$ to $G$, and in addition $g(000)=f(011)=f(101)$. Note that by the symmetry of $Q_3$ this is the same as the number of maps $h:\{110,111\}\rightarrow V(G)$ such that $f$ and $h$ together induce a homomorphism from $Q_3[\{010,011,100,101,110,111\}]$ to $G$, and in addition $h(110)=f(011)=f(101)$.

Now note that
$$\hom(Q_3,G;\{000,011,101\})=\sum_{f} \alpha_f \beta_f,$$
where the summation is over all homomorphisms $f:Q_3[\{010,011,100,101\}]\rightarrow G$ such that $f(101)=f(011)$. Indeed, $\alpha_f \beta_f$ is the number of suitable homomorphisms extending $f$. Similarly,
$$\hom(Q_3,G;\{000,011,101,110\})=\sum_{f} \beta_f^2$$ and
$$\hom(Q_3,G)\geq \sum_{f} \alpha_f^2.$$
The required inequality follows from the Cauchy-Schwarz inequality.
\end{proof}

It is straightforward to combine the two inequalities in Lemma \ref{lem:cube ineqs} to conclude the following.

\begin{corollary} \label{cor:cube final ineq}
    For any graph $G$, we have
    $$\hom(Q_3,G;\{000,011,101,110\})\geq \frac{\hom(Q_3,G;\{000,011\})^4}{\hom(Q_3,G)^3}.$$
\end{corollary}

We say that a graph $G$ is $K$-almost regular if $\Delta(G)\leq K\delta(G)$. We are now in a position to prove Proposition \ref{prop:cube} for the special case of bipartite almost regular graphs.

\begin{proposition} \label{prop:cube_almost reg}
    For any $K>0$, there are positive constants $c=c(K)$ and $C=C(K)$ such that any bipartite $K$-almost regular $n$-vertex graph with edge density $p\geq Cn^{-3/8}$ contains at least $cn^8p^{12}$ copies of $Q_3$.
\end{proposition}

As we have mentioned, the proof uses the fact that $Q_3$ satisfies Sidorenko's conjecture. Sidorenko's conjecture states that for every bipartite graph $H$ and $n$-vertex graph $G$ with edge density $p$, we have $\hom(H,G)\geq n^{v(H)}p^{e(H)}$. We say that a graph $H$ satisfies Sidorenko's conjecture if this inequality holds for every $G$. Hatami proved that $Q_d$ satisfies Sidorenko's conjecture for every $d$. 

\begin{lemma}[Hatami \cite{Hat10}] \label{lem:cube sidorenko}
    Let $d$ be a positive integer. Then any $n$-vertex graph $G$ with edge density $p$ satisfies $\hom(Q_d,G)\geq  n^{2^d}p^{d2^{d-1}}$.
\end{lemma}

\begin{proof}[Proof of Proposition \ref{prop:cube_almost reg}]
Let $C$ be sufficiently large and let $G$ be a bipartite $K$-almost regular $n$-vertex graph with edge density $p\geq Cn^{-3/8}$.

Assume, for the sake of contradiction, that $\hom(Q_3,G;\{000,011\})\geq \frac{1}{24}\hom(Q_3,G)$. Then Corollary \ref{cor:cube final ineq} implies that $$\hom(Q_3,G;\{000,011,101,110\})\geq \frac{1}{24^4}\hom(Q_3,G).$$
On the other hand, observe that
$$\hom(Q_3,G;\{000,011,101,110\})\leq n(\Delta(G))^4\leq n(Kpn)^4,$$
so, using Lemma \ref{lem:cube sidorenko}, we have $$\frac{1}{24^4} n^8p^{12}\leq \frac{1}{24^4}\hom(Q_3,G)\leq n(Kpn)^4.$$ It follows that $p\leq (24^4 K^4 )^{1/8} n^{-3/8}$, which contradicts $p\geq Cn^{-3/8}$ provided that $C$ is sufficiently large.

Hence, we have $\hom(Q_3,G;\{000,011\})< \frac{1}{24}\hom(Q_3,G)$. It follows by symmetry that for any $u,v\in V(Q_3)$ of distance two,
$\hom(Q_3,G;\{u,v\})<\frac{1}{24}\hom(Q_3,G)$. Since $G$ is bipartite, the total number of non-injective homomorphic copies of $Q_3$ in $G$ is at most $\sum \hom(Q_3,G;\{u,v\})$, where the summation is over all $u$ and $v$ of distance two in $Q_3$. By the above inequality, this sum is less than $12\cdot \frac{1}{24}\hom(Q_3,G)=\hom(Q_3,G)/2$. Hence, there are at least $\hom(Q_3,G)/2$ injective homomorphic copies of $Q_3$ in $G$. Proposition \ref{prop:cube_almost reg} now follows by another application of Lemma~\ref{lem:cube sidorenko}.
\end{proof}

In order to deduce Proposition \ref{prop:cube} from Proposition \ref{prop:cube_almost reg}, we can use a regularization lemma of Jiang and Yepremyan. We remark that the first result of this kind was established by Erd\H os and Simonovits \cite{ES69} in order to bound the Tur\'an number of the (3-dimensional) cube. Roughly speaking, they showed that in bipartite Tur\'an problems, it suffices to consider almost regular host graphs. Jiang and Yepremyan extended this to supersaturation problems. While their result applies for general linear hypergraphs, we will only need it in the special case of graphs. 

\begin{lemma}[Jiang--Yepremyan {\cite[Theorem 3.3]{JY20}}] \label{lem:regularization}
    Let $0<\alpha<1$ be a real number. Let $H$ be a graph with $e(H)\geq v(H)$. There exists a real number $K=K(\alpha,H)\geq 1$ such that the following holds. Suppose that there are positive constants $c$ and $C$ (possibly depending on $H$) such that for each $n$, every $n$-vertex, $K$-almost regular bipartite graph $G$ with edge density $p\geq Cn^{-\alpha}$ has at least $cn^{v(H)}p^{e(H)}$ copies of $H$. Then there exist positive constants $c'$ and $C'$ (possibly depending on $H$) such that for each $n$, every $n$-vertex bipartite graph $G$ with edge density $p\geq C'n^{-\alpha}$ has at least $c'n^{v(H)}p^{e(H)}$ copies of $H$.
\end{lemma}

It is straightforward to deduce Proposition \ref{prop:cube} from Proposition \ref{prop:cube_almost reg} using this lemma. We will give the details in the next subsection (in a more general setting).

\subsection{Our main general result}

In this subsection, we present our main technical results. We remark that our method resembles that of Conlon and Lee from \cite{CL17} where they prove Sidorenko's conjecture for a certain class of graphs. 

Given a graph automorphism $\phi:V(H)\rightarrow V(H)$, we write $F_{\phi}=\{v\in V(H): \phi(v)=v\}$.

\begin{definition} \label{def:symmetric triples}
    Let $H$ be a connected bipartite graph. We say that vertex sets $A,B\subset V(H)$ and a graph automorphism $\phi:V(H)\rightarrow V(H)$ form a \emph{symmetric} triple if $\phi=\phi^{-1}$; $A$, $B$ and $F_{\phi}$ partition $V(H)$; $F_{\phi}$ separates $A$ and $B$; and $\phi(A)=B$.
    
    Given a further subset $R\subset V(H)$, we say that $R$ is \emph{intersecting} for a symmetric triple $(A,B,\phi)$ if all vertices of $R$ are in the same part of the bipartition of $H$, and $R$ intersects both $A\cup F_{\phi}$ and $B\cup F_{\phi}$.
\end{definition}

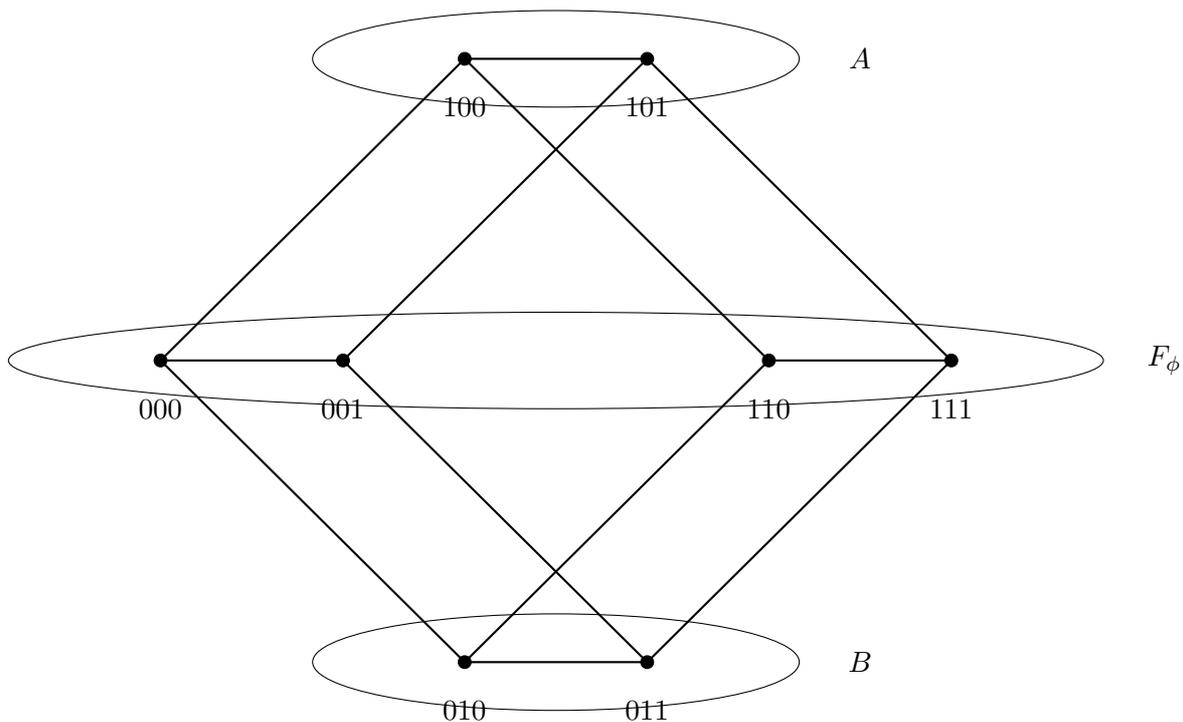
\begin{figure}
	\centering
	\begin{tikzpicture}[scale=0.8]
		\draw[fill=black](0,0)circle(3pt);
		\draw[fill=black](3,0)circle(3pt);
		\draw[fill=black](5,-5)circle(3pt);
		\draw[fill=black](8,-5)circle(3pt);
		\draw[fill=black](10,0)circle(3pt);
		\draw[fill=black](13,0)circle(3pt);
		\draw[fill=black](5,5)circle(3pt);
		\draw[fill=black](8,5)circle(3pt);

		\draw[thick](0,0)--(3,0)(5,-5)--(8,-5)(10,0)--(13,0)(5,5)--(8,5);
		\draw[thick](0,0)--(5,-5)(0,0)--(5,5)(10,0)--(5,-5)(10,0)--(5,5);
        \draw[thick](3,0)--(8,-5)(3,0)--(8,5)(13,0)--(8,-5)(13,0)--(8,5);
		
		\node at (0,-0.8)  {000};
		\node at (3,-0.8)  {001};
		\node at (5,-5.8)  {010};
		\node at (8,-5.8)  {011};
		\node at (10,-0.8)  {110};
		\node at (13,-0.8)  {111};
		\node at (5,4.2)  {100};
		\node at (8,4.2)  {101};
		
		\draw[rotate around={90:(6.5,0)}] (6.5,0) ellipse (0.8 and 9);
		\draw[rotate around={90:(6.5,5)}] (6.5,5) ellipse (0.8 and 4);
		\draw[rotate around={90:(6.5,-5)}] (6.5,-5) ellipse (0.8 and 4);
		
		\node at (11.5,5)  {$A$};
		\node at (11.5,-5)  {$B$};
		\node at (16.5,0)  {$F_{\phi}$};
	\end{tikzpicture}
	\caption{A symmetric triple $(A,B,\phi)$} \label{fig:symmetric triple}
\end{figure}

\begin{example}
Let $H$ be the $3$-dimensional cube, as depicted on Figure \ref{fig:symmetric triple}. Let $\phi$ be the automorphism which swaps the first digit with the second digit, i.e. which maps $abc$ to $bac$. Let $A=\{100,101\}$ and let $B=\{010,011\}$. Then $(A,B,\phi)$ is a symmetric triple. Moreover, if $R=\{000,011\}$, then $R$ is intersecting for $(A,B,\phi)$.
\end{example}

\begin{definition}
    Let $H$ be a connected bipartite graph, let $(A,B,\phi)$ be a symmetric triple and let $R\subset V(H)$. Then
    $$\psi_{A,B,\phi}(R)=(R\cap(A\cup F_{\phi}))\cup \phi(R\cap A).$$
    Informally, we keep all members of $R$ that are in $A\cup F_{\phi}$, but replace $R\cap B$ by $\phi(R\cap A)$.
\end{definition}

\begin{remark} \label{rem:nonempty}
If $(A,B,\phi)$ is a symmetric triple, then so is $(B,A,\phi)$, and if $R$ is intersecting for $(A,B,\phi)$, then it is intersecting also for $(B,A,\phi)$. Moreover, in this case $\psi_{A,B,\phi}(R)\neq \emptyset$. Also note that all vertices in $\psi_{A,B,\phi}(R)$ are in the same part of the bipartition of $H$ as $R$.
\end{remark}

We can now state the main technical lemma, which generalizes the inequalities from Lemma~\ref{lem:cube ineqs}.

\begin{lemma} \label{lem:general ineq one step}
    Let $H$ be a connected bipartite graph, let $(A,B,\phi)$ be a symmetric triple and let $R$ be an intersecting set for $(A,B,\phi)$. Then, for any graph $G$, we have
    $$\hom(H,G;R)^2\leq \hom(H,G;\psi_{A,B,\phi}(R))\hom(H,G;\psi_{B,A,\phi}(R)).$$
    In particular, for any graph $G$,
    $$\hom(H,G;R)^2\leq \hom(H,G;\psi_{A,B,\phi}(R))\hom(H,G).$$
\end{lemma}

\begin{proof}
Let $v\in V(G)$ and let $f:H[F_{\phi}]\rightarrow G$ be a homomorphism which maps each vertex in $R\cap F_{\phi}$ to $v$. Let $\alpha_{v,f}$ be the number of maps $g:A\rightarrow V(G)$ such that $f$ and $g$ together induce a homomorphism from $H[A\cup F_{\phi}]$ to $G$ and which maps each vertex in $R\cap A$ to $v$. Finally, let $\beta_{v,f}$ be the number of maps $h:B\rightarrow V(G)$ such that $f$ and $h$ together induce a homomorphism from $H[B\cup F_{\phi}]$ to $G$ and which map each vertex in $R\cap B$ to $v$.

Note that the number of homomorphisms $\theta:H\rightarrow G$ which extend $f$ and which map $R$ to $v$ is precisely $\alpha_{v,f}\beta_{v,f}$. Indeed, there are $\alpha_{v,f}$ ways to chose $\theta|_A$, there are $\beta_{v,f}$ ways to choose $\theta|_B$ and since there are no edges in $H$ between $A$ and $B$, any pair gives a suitable choice.
Hence,
$$\hom(H,G;R)=\sum_{v,f} \alpha_{v,f}\beta_{v,f},$$
where the summation is over all $v$ and $f$ as above.
Observe that, by the properties of a symmetric triple, $h\mapsto g\coloneqq h\circ \phi$ is a bijection (with inverse $g\mapsto h\coloneqq g\circ \phi$) between
\begin{itemize}
    \item maps $h:B\rightarrow V(G)$ with the property that $f$ and $h$ together induce a homomorphism from $H[B\cup F_{\phi}]$ to $G$ and which map $\phi(R\cap A)$ to $v$ and
    \item maps $g:A\rightarrow V(G)$ with the property that $f$ and $g$ together induce a homomorphism from $H[A\cup F_{\phi}]$ to $G$ and which map $R\cap A$ to $v$.
\end{itemize}
(Indeed, if $f$ and $h$ together induce a homomorphism from $H[B\cup F_{\phi}]$ to $G$, then $f\circ \phi$ and $h\circ \phi$ together induce a homomorphism from $H[\phi^{-1}(B\cup F_{\phi})]=H[A\cup F_{\phi}]$ to $G$, but $f\circ \phi=f$ on $F_{\phi}$.)

Therefore, the number of maps $h:B\rightarrow V(G)$ with the property that $f$ and $h$ together induce a homomorphism from $H[B\cup F_{\phi}]$ to $G$ and which map $\phi(R\cap A)$ to $v$ is precisely $\alpha_{v,f}$. Hence, using that $\psi_{A,B,\phi}(R)\neq \emptyset$, we have
$$\hom(H,G;\psi_{A,B,\phi}(R))=\sum_{v,f} \alpha_{v,f}^2,$$
where the summation is over all pairs $v,f$ as above. Similarly, we obtain
$$\hom(H,G;\psi_{B,A,\phi}(R))=\sum_{v,f} \beta_{v,f}^2$$
and we are done by the Cauchy-Schwarz inequality.
\end{proof}

We can now describe the main condition that a graph $H$ needs to satisfy in order for our method to apply.

\begin{definition} \label{def:reflective}
    Let $H$ be a connected bipartite graph with parts $X_1$ and $X_2$. We say that $H$ is \emph{reflective} if the following holds. Let $R\subset X_i$ be a set of size two for some $i\in \{1,2\}$. Then there exists a sequence of symmetric triples $(A_j,B_j,\phi_j)$ for $j=0,1,\dots,m-1$ and intersecting sets $R_j$ for $(A_j,B_j,\phi_j)$ such that $R_0=R$, $R_m=X_i$ and $R_{j+1}=\psi_{A_j,B_j,\phi_j}(R_j)$ for all $0\leq j\leq m-1$.
\end{definition}

\begin{remark} \label{rem:monotone}
Observe that if $R$ is intersecting for a symmetric triple $(A,B,\phi)$ and $S\supset R$, then $S$ is also intersecting for $(A,B,\phi)$ and $\psi_{A,B,\phi}(S)\supset \psi_{A,B,\phi}(R)$.
Hence, $H$ is reflective if for each $R\subset X_i$ of size two for some $i\in \{1,2\}$, there exists a sequence of symmetric triples $(A_j,B_j,\phi_j)$ for $j=0,1,\dots,m-1$ and intersecting sets $R_j$ for $(A_j,B_j,\phi_j)$ such that $R_0=R$, $R_m=X_i$ and $R_{j+1}\subset \psi_{A_j,B_j,\phi_j}(R_j)$ for all $0\leq j\leq m-1$.
\end{remark}

The following lemma generalizes Corollary \ref{cor:cube final ineq}.

\begin{lemma} \label{lem:general final ineq}
    Let $H$ be a reflective connected bipartite graph and let $R\subset X$ be a set of size two, where $X$ is one of the parts of $H$. Then there is a positive integer $s$ such that for every graph $G$, we have
    $$\hom(H,G;X)\geq \frac{\hom(H,G;R)^s}{\hom(H,G)^{s-1}}.$$
\end{lemma}

\begin{proof}
Since $H$ is reflective, we can choose a sequence of symmetric triples $(A_j,B_j,\phi_j)$ for $j=0,1,\dots,m-1$ and intersecting sets $R_j$ for $(A_j,B_j,\phi_j)$ such that $R_0=R$, $R_m=X$ and $R_{j+1}=\psi_{A_j,B_j,\phi_j}(R_j)$ for all $0\leq j\leq m-1$. By Lemma \ref{lem:general ineq one step}, we have
$$\hom(H,G;R_j)^2\leq \hom(H,G;R_{j+1})\hom(H,G)$$
for each $0\leq j\leq m-1$. It is easy to see that this implies that
$$\hom(H,G;X)=\hom(H,G;R_m)\geq \frac{\hom(H,G;R_0)^{2^m}}{\hom(H,G)^{2^m-1}}=\frac{\hom(H,G;R)^{2^m}}{\hom(H,G)^{2^m-1}},$$
so we may take $s=2^m$.
\end{proof}

The next proposition is our main result restricted to almost regular bipartite host graphs.

\begin{proposition} \label{prop:main_almost regular}
    Let $H$ be a reflective connected bipartite graph which satisfies Sidorenko's conjecture. Let $K\geq 1$ be a real number. Then there are positive constants $c=c(H)$ and $C=C(H,K)$ such that if $G$ is a $K$-almost regular bipartite $n$-vertex graph with edge density $p$ satisfying $n^{v(H)}p^{e(H)}\geq Cn(pn)^t$, where $t$ is the size of the larger part in the bipartition of $H$, then $G$ contains at least $cn^{v(H)}p^{e(H)}$ copies of $H$.
\end{proposition}

\begin{proof}
Let $c=c(H)$ be a sufficiently small positive real and let $C=C(H,K)$ be sufficiently large. Let $G$ be a $K$-almost regular bipartite $n$-vertex graph with edge density $p$ satisfying $n^{v(H)}p^{e(H)}\geq Cn(pn)^t$, where $t$ is the size of the larger part in the bipartition of $H$. Since $H$ satisfies Sidorenko's conjecture, we have $\hom(H,G)\geq n^{v(H)}p^{e(H)}$.

\medskip

\noindent \emph{Claim.} For every $R\subset V(H)$ of size two, we have $$\hom(H,G;R)\leq \frac{\hom(H,G)}{v(H)^2}.$$

\medskip

\noindent \emph{Proof of Claim.} Suppose, for the sake of contradiction, that $$\hom(H,G;R)> \frac{\hom(H,G)}{v(H)^2}.$$
In particular, there is a homomorphism $H\rightarrow G$ which maps the two elements of $R$ to the same vertex. Hence, as $G$ is bipartite, the two elements of $R$ are in the same part of the bipartition of $H$. Let $X$ be this part. By Lemma \ref{lem:general final ineq}, we have
\begin{equation}
    \hom(H,G;X)\geq v(H)^{-2s}\hom(H,G) \label{eqn:star vs all}
\end{equation}
for some positive integer $s$ that only depends on $H$.
Since $\hom(H,G;X)\leq n\Delta(G)^{v(H)-|X|}\leq n(Kpn)^{v(H)-|X|}\leq n(Kpn)^t$ and $\hom(H,G)\geq n^{v(H)}p^{e(H)}$, equation (\ref{eqn:star vs all}) implies that
$$n(Kpn)^t\geq v(H)^{-2s}n^{v(H)}p^{e(H)}.$$
However, this contradicts the assumption that $n^{v(H)}p^{e(H)}\geq Cn(pn)^t$ and that $C$ is sufficiently large. This completes the proof of the claim. $\Box$

\medskip

Now note that the number of non-injective homomorphisms $H\rightarrow G$ is at most $\sum_R \hom(H,G;R)$, where the summation is over all $R\subset V(H)$ of size two. By the claim, this sum is at most $\binom{v(H)}{2}\cdot \frac{\hom(H,G)}{v(H)^2}\leq \frac{1}{2}\hom(H,G)$. Hence, there are at least $\frac{1}{2}\hom(H,G)$ injective homomorphisms $H\rightarrow G$, which implies that there are at least $c\hom(H,G)\geq cn^{v(H)}p^{e(H)}$ copies of $H$ in $G$, provided that $c$ is sufficiently small.
\end{proof}

We are now in a position to state and prove our main result, which follows easily from Proposition \ref{prop:main_almost regular} and Lemma \ref{lem:regularization}.

\begin{theorem} \label{thm:general supersat result}
    Let $H$ be a reflective connected bipartite graph which satisfies Sidorenko's conjecture and which is not a tree. Then there are positive constants $c=c(H)$ and $C=C(H)$ such that if $G$ is an $n$-vertex graph with edge density $p\geq Cn^{-\frac{v(H)-t-1}{e(H)-t}}$, where $t$ is the size of the larger part in the bipartition of $H$, then $G$ contains at least $cn^{v(H)}p^{e(H)}$ copies of $H$.
\end{theorem}

\begin{proof}
Let $\alpha=\frac{v(H)-t-1}{e(H)-t}$. Let $K=K(\alpha,H)$ be the constant provided by Lemma \ref{lem:regularization}. By Proposition \ref{prop:main_almost regular}, there are positive constants $c'=c'(H)$ and $C''=C''(H)$ such that if $G$ is a $K$-almost regular bipartite $n$-vertex graph with edge density $p$ satisfying $n^{v(H)}p^{e(H)}\geq C''n(pn)^t$, then $G$ has at least $c'n^{v(H)}p^{e(H)}$ copies of $H$. Now note that there is some $C'=C'(H)$ such that if $p\geq C'n^{-\alpha}$, then $n^{v(H)}p^{e(H)}\geq C''n(pn)^t$ holds. Hence, any $K$-almost regular bipartite $n$-vertex graph $G$ with edge density $p\geq C'n^{-\alpha}$ contains at least $c'n^{v(H)}p^{e(H)}$ copies of $H$. It follows by Lemma \ref{lem:regularization} that there are positive constants $c=c(H)$ and $C=C(H)$ such that if $G$ is a bipartite $n$-vertex graph with edge density $p\geq Cn^{-\alpha}$, then $G$ contains at least $cn^{v(H)}p^{e(H)}$ copies of $H$. This proves the theorem for all bipartite host graphs $G$. The general case follows easily by noting that any graph $G$ has a bipartite subgraph with at least half of the edges of $G$.
\end{proof}

We also state a simple corollary of our main result.

\begin{theorem} \label{thm:regular version}
    Let $H$ be a $d$-regular, reflective, connected bipartite graph which satisfies Sidorenko's conjecture and which is not $K_{d,d}$. Then there is some $\eps=\eps(H)>0$ such that $\ex(n,H)=O(n^{2-1/d-\eps})$.
\end{theorem}

\begin{proof}
By Theorem \ref{thm:general supersat result}, there are positive constants $c=c(H)$ and $C=C(H)$ such that if $G$ is an $n$-vertex graph with edge density $p\geq Cn^{-\frac{v(H)-t-1}{e(H)-t}}$, where $t$ is the size of the larger part in the bipartition of $H$, then $G$ contains at least $cn^{v(H)}p^{e(H)}$ copies of $H$. This implies that
$$\ex(n,H)=O(n^{2-\frac{v(H)-t-1}{e(H)-t}}).$$
Since $H$ is $d$-regular, we have $t=v(H)/2$ and $e(H)=dv(H)/2$, so
$$2-\frac{v(H)-t-1}{e(H)-t}=2-\frac{v(H)/2-1}{dv(H)/2-v(H)/2}<2-1/d,$$
where the last inequality follows from $v(H)/2>d$ (which is true since $H$ is $d$-regular and $H\neq K_{d,d}$). This completes the proof.
\end{proof}

\subsection{Hypercubes}

In this subsection we show that any hypercube is reflective and use this to deduce Theorem \ref{thm:cube supersaturation}.

\begin{lemma} \label{lem:hypercube reflective}
    For any $d\geq 3$, the hypercube $Q_d$ is reflective.
\end{lemma}

\begin{proof}
Identify $Q_d$ with $\{0,1\}^d$. Let $Q_d(0)=\{\vx\in Q_d: \sum_i x_i\equiv 0 \mod 2\}$. By the symmetry of the cube, it suffices to prove that for any $R\subset Q_d(0)$ of size two, there exists a sequence of symmetric triples $(A_j,B_j,\phi_j)$ for $j=0,\dots,m-1$ and intersecting sets $R_j$ for $(A_j,B_j,\phi_j)$ such that $R_0=R$, $R_m=Q_d(0)$ and $R_{j+1}=\psi_{A_j,B_j,\phi_j}(R_j)$ for all $0\leq j\leq m-1$.

First we prove this in the special case where the two vertices of $R$ has distance two in $Q_d$. By the symmetry of the cube, we may assume that $R=\{(0,0,0,\dots,0),(1,1,0,\dots,0)\}$.

For every $0\leq k\leq d$, let $$S_k=\{\vx\in Q_d(0): x_i=0 \text{ for all } i>k\}.$$
Also, for $1\leq k\leq d-1$, let $$T_k=\{\vx\in Q_d(0): x_i=0 \text{ for all } i>k+1 \text{ and } (x_{k},x_{k+1})\neq (1,1)\}.$$
Observe that $R=S_2$ and that $Q_d(0)=S_d$.

\medskip
\noindent \emph{Claim.} For every $2\leq k\leq d-1$, there is a symmetric triple $(A,B,\phi)$ such that $S_k$ is intersecting for $(A,B,\phi)$ and $T_{k}=\psi_{A,B,\phi}(S_k)$. Also, there is a symmetric triple $(A',B',\phi')$ such that $T_{k}$ is intersecting for $(A',B',\phi')$ and $S_{k+1}=\psi_{A',B',\phi'}(T_{k})$.

\medskip

\noindent \emph{Proof of Claim.} We start with the first assertion. Let $\phi$ be the automorphism of $Q_d$ which swaps the $k$th and the $(k+1)$th coordinate of each element in $Q_d$. Let $A=\{\vx\in Q_d: x_k=1, x_{k+1}=0\}$ and let $B=\{\vx\in Q_d: x_k=0, x_{k+1}=1\}$. Clearly, $\phi^{-1}=\phi$; $A$, $B$ and $F_{\phi}$ partition $Q_d$; $F_{\phi}$ separates $A$ and $B$; and $\phi(A)=B$. Moreover, $S_k$ is intersecting for $(A,B,\phi)$ since $(0,0,\dots,0)\in S_k\cap F_{\phi}$. Recall that $\psi_{A,B,\phi}(S_k)=(S_k\cap (A\cup F_{\phi}))\cup \phi(S_k\cap A)$. Hence,
\begin{align*}
    \psi_{A,B,\phi}(S_k)
    &=(S_k\cap \{\vx\in Q_d: (x_k,x_{k+1})\neq (0,1)\})\cup \phi(S_k\cap \{\vx\in Q_d: (x_k,x_{k+1})=(1,0)\}) \\
    &=S_k\cup \phi(S_k\cap \{\vx\in Q_d: x_k=1\}) \\
    &=T_{k}.
\end{align*}
For the second assertion, let $\phi'$ be the automorphism of $Q_d$ defined by
$$\phi'\left((x_1,x_2,\dots,x_d)\right)=\left(x_1,\dots,x_{k-1},1-x_{k+1},1-x_k,x_{k+2},\dots,x_d\right).$$
Let $A'=\{\vx\in Q_d: x_k=0, x_{k+1}=0\}$ and let $B'=\{\vx\in Q_d: x_k=1, x_{k+1}=1\}$. Clearly, $(\phi')^{-1}=\phi'$; $A'$, $B'$ and $F_{\phi'}$ partition $Q_d$; $F_{\phi'}$ separates $A'$ and $B'$; and $\phi'(A')=B'$. Moreover, $T_{k}$ is intersecting for $(A',B',\phi')$ since $T_k\cap F_{\phi'}$ contains the vector whose only non-zero coordinates are the first and the $k$th coordinate. Finally,
\begin{align*}
    \psi_{A',B',\phi'}(T_k)
    &=(T_k\cap \{\vx\in Q_d: (x_k,x_{k+1})\neq (1,1)\})\cup \phi'(T_k\cap \{\vx\in Q_d: (x_k,x_{k+1})=(0,0)\}) \\
    &=T_k\cup \phi'(T_k\cap \{\vx\in Q_d: (x_k,x_{k+1})=(0,0)\}) \\
    &=S_{k+1},
\end{align*}
which completes the proof of the claim. $\Box$

\medskip

The claim implies that whenever $R\subset Q_d(0)$ consists of two elements of distance two in $Q_d$, there exists a sequence of symmetric triples $(A_j,B_j,\phi_j)$ for $j=0,\dots,m-1$ and intersecting sets $R_j$ for $(A_j,B_j,\phi_j)$ such that $R_0=R$, $R_m=Q_d(0)$ and $R_{j+1}=\psi_{A_j,B_j,\phi_j}(R_j)$ for all $0\leq j\leq m-1$. It is therefore sufficient (by Remark \ref{rem:monotone}) to prove that if $P\subset Q_d(0)$ has size two, then there is a symmetric triple $(A,B,\phi)$ such that $P$ is intersecting for $(A,B,\phi)$ and $\psi_{A,B,\phi}(P)$ contains two elements of distance two. Let $P=\{u,v\}$. We consider two cases.

\medskip

\noindent \emph{Case 1.} $u$ and $v$ are not antipodal points of $Q_d$. Without loss of generality, we may assume that $u=(0,0,\dots,0)$ (so $v\neq (1,1,\dots,1)$ by assumption). In particular, there exist some $1\leq i<j\leq d$ such that $v_i\neq v_j$. Let $\phi$ be the automorphism of $Q_d$ which swaps the $i$th and the $j$th coordinate. Let $A=\{\vx\in Q_d: x_i=v_i, x_j=v_j\}$ and let $B=\{\vx\in Q_d: x_i=1-v_i, x_j=1-v_j\}$. Note that $(A,B,\phi)$ is a symmetric triple and $u\in F_{\phi}$, so $P$ is intersecting for $(A,B,\phi)$. Now $\psi_{A,B,\phi}(P)$ contains both $v$ and $\phi(v)$, so it contains two elements of distance two in $Q_d$.

\medskip
\noindent \emph{Case 2.} $u$ and $v$ are antipodal in $Q_d$. Without loss of generality, $u=(0,0,\dots,0)$ and $v=(1,1,\dots,1)$. Let $\phi$ be the automorphism of $Q_d$ which maps $(x_1,x_2,x_3,\dots,x_d)$ to $(1-x_2,1-x_1,x_3,\dots,x_d)$. Let $A=\{\vx\in Q_d: x_1=0, x_2=0\}$ and let $B=\{\vx\in Q_d: x_1=1, x_2=1\}$. Then $(A,B,\phi)$ is a symmetric triple and $u\in A$, $v\in B$, so $P$ is intersecting for $(A,B,\phi)$. Now $\psi_{A,B,\phi}(P)$ contains both $u$ and $\phi(u)$, so it contains two elements of distance two in $Q_d$.
\end{proof}

We can now easily deduce Theorem \ref{thm:cube supersaturation}.

\begin{proof}[Proof of Theorem \ref{thm:cube supersaturation}]
Let $d\geq 3$ be an integer. By Lemma~\ref{lem:hypercube reflective}, $Q_d$ is reflective. By Lemma~\ref{lem:cube sidorenko}, it satisfies Sidorenko's conjecture. Hence, by Theorem \ref{thm:general supersat result}, there are positive constants $c=c(d)$ and $C=C(d)$ such that if $G$ is an $n$-vertex graph with edge density $p\geq Cn^{-\frac{v(Q_d)-t-1}{e(Q_d)-t}}$, where $t$ is the size of a part of the bipartition of $Q_d$, then $G$ contains at least $cn^{v(Q_d)}p^{e(Q_d)}$ copies of $Q_d$. The result follows by noting that $v(Q_d)=2^d$, $e(Q_d)=d2^{d-1}$ and $t=2^{d-1}$.
\end{proof}

\subsection{Bipartite Kneser graphs}

In this subsection we prove that bipartite Kneser graphs (see Definition \ref{def:set graph}) are reflective and use this to deduce Theorem \ref{thm:set graph}.

\begin{lemma} \label{lem:set graph reflective}
    For any $1\leq \ell<k/2$, the graph $H_{\ell,k}$ from Definition \ref{def:set graph} is reflective.
\end{lemma}

\begin{proof}
By the symmetry of the two parts of $H_{\ell,k}$, it suffices to prove that if $R\subset [k]^{(\ell)}$ is a set of size two, then there exists a sequence of symmetric triples $(A_j,B_j,\phi_j)$ for $j=0,1,\dots,m-1$ and intersecting sets $R_j$ for $(A_j,B_j,\phi_j)$ such that $R_0=R$, $R_m=[k]^{(\ell)}$ and $R_{j+1}=\psi_{A_j,B_j,\phi_j}(R_j)$ for all $0\leq j\leq m-1$.

We first prove this for sets of the form $R=\{S,T\}$, where $|S\Delta T|=1$. Without loss of generality, we may assume that $S=[\ell]$ and $T=[\ell-1]\cup \{\ell+1\}$.

For each $1\leq i<j\leq k$, let
$$C_{i,j}=\{P\in V(H_{\ell,k}): i\in P,j\not \in P\}$$
and let
$$D_{i,j}=\{P\in V(H_{\ell,k}): i\not\in P,j\in P\}.$$
Let $\varphi_{i,j}$ be the automorphism of $H_{\ell,k}$ that swaps $i$ and $j$, i.e., which is defined as
$$\varphi_{i,j}(P)=
\begin{cases}
    (P\cup \{j\})\setminus \{i\} \text{ if } P\in C_{i,j} \\
    (P\cup \{i\})\setminus \{j\} \text{ if } P\in D_{i,j} \\
    P \text{ otherwise.}
\end{cases}
$$
\sloppy Note that $(C_{i,j},D_{i,j},\varphi_{i,j})$ is a symmetric triple. Define the following sequence:
$\phi_0=\varphi_{\ell,\ell+1}$, $\phi_1=\varphi_{\ell,\ell+2}$, \dots, $\phi_{k-\ell-1}=\varphi_{\ell,k}$, $\phi_{k-\ell}=\varphi_{\ell-1,\ell}$, $\phi_{k-\ell+1}=\varphi_{\ell-1,\ell+1}$, \dots, $\phi_{2k-2\ell}=\varphi_{\ell-1,k}$, $\phi_{2k-2\ell+1}=\varphi_{\ell-2,\ell-1}$, $\phi_{2k-2\ell+2}=\varphi_{\ell-2,\ell}$, \dots, $\phi_{\binom{k}{2}-\binom{k-\ell}{2}-2}=\varphi_{1,k-1}$, $\phi_{\binom{k}{2}-\binom{k-\ell}{2}-1}=\varphi_{1,k}$. Similarly, for $0\leq t\leq \binom{k}{2}-\binom{k-\ell}{2}-1$, let $A_t=C_{i,j}$ and $B_t=D_{i,j}$ for those $i,j$ with $\phi_t=\varphi_{i,j}$.

Now for all $0\leq t\leq \binom{k}{2}-\binom{k-\ell}{2}-1$, let $R_{t+1}=\psi_{A_t,B_t,\phi_t}(R_t)$.

\medskip

\noindent \emph{Claim.} Let $0\leq t\leq \binom{k}{2}-\binom{k-\ell}{2}-1$. Assume that $\phi_t=\varphi_{i,j}$. Then $$R_{t+1}\supset \{P\in [k]^{(\ell)}: P\supset [i-1] \text{ and } P\cap \{i,i+1,\dots,j\}\neq \emptyset\}.$$

\medskip

\noindent \emph{Proof of Claim.} We use induction on $t$. For $t=0$, note that $(i,j)=(\ell,\ell+1)$ and $$R_1=\psi_{A_0,B_0,\phi_0}(R)=\{S,T\}=\{P\in [k]^{(\ell)}: P\supset [i-1] \text{ and } P\cap \{i,i+1,\dots,j\}\neq \emptyset\}.$$
Assume now that we have already proved that for some $t$ with $\phi_t=\varphi_{i,j}$, we have $$R_{t+1}\supset \{P\in [k]^{(\ell)}: P\supset [i-1] \text{ and } P\cap \{i,i+1,\dots,j\}\neq \emptyset\}.$$
There are two cases. The first case is where $j=k$. Then $\phi_{t+1}=\varphi_{i-1,i}$. Let $P\in [k]^{(\ell)}$ satisfy $P\supset [i-2]$ and $P\cap \{i-1,i\}\neq \emptyset$.

If $P\in D_{i-1,i}=B_{t+1}$, then $\varphi_{i-1,i}(P)\supset [i-1]$, so $\varphi_{i-1,i}(P)\in R_{t+1}$. Also, $\varphi_{i-1,i}(P)\in C_{i-1,i}=A_{t+1}$, so $\varphi_{i-1,i}(P)\in R_{t+1}\cap A_{t+1}$, which implies that $P\in \varphi_{i-1,i}(R_{t+1}\cap A_{t+1})$ as $\varphi_{i-1,i}$ is an involution.

Else (i.e. if $P\not \in D_{i-1,i}$) $P\supset [i-1]$, so $P\in R_{t+1}\cap (A_{t+1}\cup F_{\phi_{t+1}})$. Hence, in both cases, $P\in \psi_{A_{t+1},B_{t+1},\phi_{t+1}}(R_{t+1})=R_{t+2}$. Thus,
$$R_{t+2}\supset \{P\in [k]^{(\ell)}: P\supset [i-2] \text{ and } P\cap \{i-1,i\}\neq \emptyset\},$$
completing the induction step.

The second case is where $j\neq k$. Then $\phi_{t+1}=\varphi_{i,j+1}$. Let $P\in [k]^{(\ell)}$ satisfy $P\supset [i-1]$ and $P\cap \{i,i+1,\dots,j+1\}\neq \emptyset$. If $P\in D_{i,j+1}=B_{t+1}$, then $\varphi_{i,j+1}(P)\supset [i]$, so $\varphi_{i,j+1}(P)\in R_{t+1}$. Else, $P\cap \{i,i+1,\dots,j\}\neq \emptyset$, so $P\in R_{t+1}$. Hence, in both cases, $P\in \psi_{A_{t+1},B_{t+1},\phi_{t+1}}(R_{t+1})=R_{t+2}$. Thus,
$$R_{t+2}\supset \{P\in [k]^{(\ell)}: P\supset [i-1] \text{ and } P\cap \{i,i+1,\dots,j+1\}\neq \emptyset\},$$
completing the induction step and the proof of the claim. $\Box$

\medskip

By the claim, for $m=\binom{k}{2}-\binom{k-\ell}{2}$, we have $R_m=[k]^{(\ell)}$, so it remains to show that for each~$t$, $R_t$ is intersecting for $(A_t,B_t,\phi_t)$. This is clear for $t=0$, so let $t>0$. Let $\phi_t=\varphi_{i,j}$. Again we consider two cases. If $i=\ell$, then $j\geq \ell+2$ and $\phi_{t-1}=\varphi_{\ell,j-1}$, so by the claim we have
$$R_{t}\supset \{P\in [k]^{(\ell)}: P\supset [\ell-1] \text{ and } P\cap \{\ell,\ell+1,\dots,j-1\}\neq \emptyset\}.$$
Hence, $[\ell-1]\cup \{\ell+1\}\in R_t$, so $R_t\cap F_{\phi_t}\neq \emptyset$. In particular, $R_t$ is intersecting for $(A_t,B_t,\phi_t)$. The other case is $i<\ell$. In this case either $\phi_{t-1}=\varphi_{i',j'}$ for some $i'<\ell$, or $\phi_{t-1}=\varphi_{\ell,k}$. Either way, the claim implies that
$$R_t\supset \{P\in [k]^{(\ell)}: P\supset [\ell-1]\}.$$
Hence, $R_t$ has an element which contains both $i$ and $j$, so $R_t\cap F_{\phi_t}\neq \emptyset$. In particular, $R_t$ is intersecting for $(A_t,B_t,\phi_t)$.

We have proved that if $R\subset [k]^{(\ell)}$ consists of two sets differing by one element, then there exists a sequence of symmetric triples $(A_j,B_j,\phi_j)$ for $j=0,1,\dots,m-1$ and intersecting sets $R_j$ for $(A_j,B_j,\phi_j)$ such that $R_0=R$, $R_m=[k]^{(\ell)}$ and $R_{j+1}=\psi_{A_j,B_j,\phi_j}(R_j)$ for all $0\leq j\leq m-1$. To complete the proof, it suffices to prove that for any $R\subset [k]^{(\ell)}$ of size two, there is a symmetric triple $(A,B,\phi)$ such that $R$ is intersecting for $(A,B,\phi)$ and $\psi_{A,B,\phi}(R)$ contains two sets which differ by one element. Let $R=\{S,T\}$. Let $i\in S\setminus T$ and let $j\in [k]\setminus (S\cup T)$ (which exists since $|S\cup T|\leq 2\ell<k$). Without loss of generality, let us assume that $i<j$. Now let $\phi=\varphi_{i,j}$, $A=C_{i,j}$ and $B=D_{i,j}$. Then $T\in F_{\phi}$, so $R$ is intersecting for $(A,B,\phi)$. Moreover, $\psi_{A,B,\phi}(R)$ contains both $S$ and $\varphi_{i,j}(S)$, so it contains two sets which differ by one element. This completes the proof.
\end{proof}

The other condition that we need to check for our Theorem \ref{thm:regular version} to apply is that $H_{\ell,k}$ satisfies Sidorenko's conjecture. This was proved by Conlon and Lee \cite[Theorem 1.1]{CL17}.

\begin{lemma}[Conlon--Lee \cite{CL17}] \label{lem:set graph Sidorenko}
    For any $1\leq \ell<k/2$, the graph $H_{\ell,k}$ from Definition \ref{def:set graph} satisfies Sidorenko's conjecture.
\end{lemma}

Notice that Theorem \ref{thm:set graph} follows from Theorem \ref{thm:regular version}, Lemma \ref{lem:set graph reflective} and Lemma \ref{lem:set graph Sidorenko}.

\section{Rainbow Tur\'an number of cycles} \label{sec:rainbow cycles}

In this section, we prove Theorems \ref{thm:rainbow cycle} and \ref{thm:almost rainbow cycle}. As before, we establish certain inequalities between various homomorphism counts. However, we can no longer assume freely that the host graph is almost regular because it is too sparse for the regularization method to work. Instead, we introduce weights for our cycles, and count these weighted homomorphic cycles.

\begin{definition}
    Let $k$ be a positive integer and let $G$ be a graph. The \emph{weight} of an edge $uv$ is defined to be $$w(uv)=\frac{1}{d_G(u)^{1/2}d_G(v)^{1/2}}.$$
    Now the weight of a walk $P=(u_0,u_1,\dots,u_k)$ is defined to be the product of the weights of the edges in it, that is, $$w(P)=\frac{1}{d_G(u_0)^{1/2}d_G(u_k)^{1/2}\prod_{i=1}^{k-1} d_G(u_i)}.$$
    Similarly, the weight of a homomorphic cycle $C=(u_0,u_1,\dots,u_{2k-1})$ is
    $$w(C)=\frac{1}{\prod_{i=0}^{2k-1} d_G(u_i)}.$$
    Finally, let $h_{2k}$ be the sum of the weights of all homomorphic cycles of length $2k$ in $G$ (here, a homomorphic cycle of length $2$ is just an edge with labelled endpoints).
\end{definition}

The next lemma can be viewed as a weighted variant of Sidorenko's conjecture for even cycles.

\begin{lemma} \label{lem:weighted cycle count}
    For any graph $G$ and any positive integer $k$, we have $h_{2k}\geq 1$.
\end{lemma}

\begin{proof}
Let $A$ be the matrix whose rows and columns are labelled by $V(G)$ and which has entries
$$A_{u,v}=
\begin{cases}
    w(uv)=\frac{1}{d_G(u)^{1/2}d_G(v)^{1/2}} &\text{ if } uv\in E(G) \\
    0 &\text{ otherwise.}
\end{cases}
$$
Observe that $h_{2k}=\mathrm{tr}(A^{2k})$. Let $\vx\in \mathbb{R}^{V(G)}$ be the vector with $\vx_u=d_G(u)^{1/2}$. Then
$$(A\vx)_u=\sum_{v\in V(G)} A_{uv}\vx_v=\sum_{v\in N_G(u)} \frac{1}{d_G(u)^{1/2}d_G(v)^{1/2}}d_G(v)^{1/2}=d_G(u)^{1/2}.$$
Hence, $A\vx=\vx$, so $1$ is an eigenvalue of $A$. Writing $\lambda_1,\dots,\lambda_n$ for the eigenvalues of $A$ (which are real numbers since $A$ is a symmetric matrix), we obtain $\mathrm{tr}(A^{2k})=\sum_{i=1}^n \lambda_i^{2k}\geq 1$, completing the proof.
\end{proof}

An interpretation of Lemma \ref{lem:weighted cycle count} is that if we choose a vertex uniformly at random in an $n$-vertex graph and start a random walk (choosing each neighbour with the same probability), then the probability of ending up at the starting vertex after $2k$ steps is at least $1/n$. Results from which this follows already exist in the literature on random walks (see, e.g., Proposition~10.25 in \cite{markovchain}), but since our proof is very short, we included it for the sake of completeness. There is another related result in \cite{alonspencer} (see the ``Probabilistic lens: Random walks''). There it is shown, using the Cauchy-Schwarz inequality, that under the extra assumption that the graph is vertex-transitive, for any two vertices $u$ and $v$, the probability that a random walk of length $2k$ starting from $u$ ends at $u$ is at least as large as the probability that it ends at $v$.

\medskip

In what follows, indices are considered modulo $2k$, e.g. $u_{2k}=u_0$.

\begin{definition}
    Given a graph $G$ with an edge-colouring $c:E(G)\rightarrow \mathcal{C}$ and positive integers $i,j,k$, let $h_{2k}(i,j)$ be the sum of the weights of homomorphic $2k$-cycles $(u_0,u_1,\dots,u_{2k-1})$ with $c(u_{i-1}u_i)=c(u_{j-1}u_j)$.
\end{definition}

The key lemma is as follows.

\begin{lemma} \label{lem:cycle CS inequality}
    For any $1\leq \ell\leq k$, we have $h_{2k}(\ell,2k)^2\leq h_{2k}(1,2k)h_{2k}(\ell,2k+1-\ell)$.
\end{lemma}

\begin{proof}
For any $u_0,u_k\in V(G)$ and $R\in \mathcal{C}$, let $\alpha(u_0,u_k,R)$ be the sum of the weights of all walks $(u_0,u_1,\dots,u_k)$ in $G$ with $c(u_{\ell-1},u_{\ell})=R$. Moreover, let $\beta(u_0,u_k,R)$ be the sum of the weights of all walks $(u_0,u_1,\dots,u_k)$ in $G$ with $c(u_0,u_1)=R$.
Note that
$$h_{2k}(\ell,2k)=\sum_{u_0,u_k\in V(G), R\in \mathcal{C}} \alpha(u_0,u_k,R)\beta(u_0,u_k,R),$$
$$h_{2k}(1,2k)=\sum_{u_0,u_k\in V(G), R\in \mathcal{C}} \beta(u_0,u_k,R)^2$$ and
$$h_{2k}(\ell,2k+1-\ell)=\sum_{u_0,u_k\in V(G), R\in \mathcal{C}} \alpha(u_0,u_k,R)^2.$$
Hence the statement of the lemma follows from the Cauchy-Schwarz inequality.
\end{proof}

\begin{lemma} \label{lem:extremal pattern}
    We have $h_{2k}(1,2k)=\max_{1\leq i<j\leq 2k} h_{2k}(i,j)$.
\end{lemma}

\begin{proof}
Choose $1\leq i'<j'\leq 2k$ such that $h_{2k}(i',j')=\max_{1\leq i<j\leq 2k} h_{2k}(i,j)$. Trivially, we have $h_{2k}(i',j')=h_{2k}(i'+t,j'+t)$ for every positive integer $t$ (here, as before, indices are considered modulo $2k$). Hence, there is some $1\leq \ell\leq k$ such that $h_{2k}(\ell,2k)=h_{2k}(i',j')=\max_{1\leq i<j\leq 2k} h_{2k}(i,j)$.

Now by Lemma \ref{lem:cycle CS inequality}, we have
$$h_{2k}(\ell,2k)^2\leq h_{2k}(1,2k)h_{2k}(\ell,2k+1-\ell)\leq h_{2k}(1,2k)\max_{1\leq i<j\leq 2k} h_{2k}(i,j)=h_{2k}(1,2k)h_{2k}(\ell,2k).$$
Hence, $$h_{2k}(1,2k)\geq h_{2k}(\ell,2k)=\max_{1\leq i<j\leq 2k} h_{2k}(i,j),$$
as desired.
\end{proof}

\begin{lemma} \label{lem:step down}
    For any properly edge-coloured graph $G$ with $\delta(G)>0$ and integer $k\geq 2$, we have $$h_{2k}(1,2k)\leq \frac{h_{2k-2}}{\delta(G)}.$$
\end{lemma}

\begin{proof}
Let $C=(u_0,u_1,\dots,u_{2k-1})$ be a homomorphic $2k$-cycle in $G$ with the property that $c(u_0 u_1)=c(u_{2k-1} u_0)$. Since $c$ is a proper colouring, we have $u_1=u_{2k-1}$. This means that $C'=(u_1,u_2,\dots,u_{2k-2})$ is a homomorphic $(2k-2)$-cycle. Note that $w(C)=\frac{w(C')}{d_G(u_0)d_G(u_1)}\leq \frac{w(C')}{\delta(G)d_G(u_1)}$. Furthermore, any homomorphic $(2k-2)$-cycle $(u_1,\dots,u_{2k-1})$ arises as $C'$ for precisely $d_G(u_1)$ choices of $C$. The desired inequality follows.
\end{proof}

\begin{lemma} \label{lem:hom ineq one step}
    Let $G$ be a properly edge-coloured graph with $\delta(G)>0$ and let $k\geq 2$ be an integer. If $G$ has no rainbow cycle, then $h_{2k}\leq \frac{2k^2}{\delta(G)} h_{2k-2}$.
\end{lemma}

\begin{proof}
Since $G$ has no rainbow cycle, we have $$h_{2k}\leq \sum_{1\leq i<j\leq 2k} h_{2k}(i,j).$$
Using Lemmas \ref{lem:extremal pattern} and \ref{lem:step down}, we have
$$\sum_{1\leq i<j\leq 2k} h_{2k}(i,j)\leq \binom{2k}{2}h_{2k}(1,2k)\leq \binom{2k}{2}\frac{h_{2k-2}}{\delta(G)},$$
which implies that the desired inequality.
\end{proof}

\begin{corollary} \label{cor:hom ineq}
    Let $G$ be an $n$-vertex properly edge-coloured graph with $\delta(G)>0$ and let $k\geq 2$ be an integer. If $G$ has no rainbow cycle, then $h_{2k}\leq (\frac{2k^2}{\delta(G)})^k n$.
\end{corollary}

\begin{proof}
By repeated applications of Lemma \ref{lem:hom ineq one step}, we obtain $h_{2k}\leq \frac{2^{k-1}(k!)^2}{\delta(G)^{k-1}}h_2$. Furthermore,
$$h_2=\sum_{u,v\in V(G): uv\in E(G)} \frac{1}{d_G(u)d_G(v)}\leq \frac{1}{\delta(G)}\sum_{u,v\in V(G): uv\in E(G)} \frac{1}{d_G(u)}=\frac{1}{\delta(G)}\sum_{u\in V(G)} 1=\frac{n}{\delta(G)},$$
which implies the result.
\end{proof}

\begin{proof}[Proof of Theorem \ref{thm:rainbow cycle}]
Let $n$ be sufficiently large and let $G$ be a properly edge-coloured $n$-vertex graph with at least $8n(\log n)^2$ edges. Then $G$ has a non-empty subgraph $G'$ with $\delta(G')\geq 8(\log n)^2$.

Assume, for contradiction that $G'$ has no rainbow cycle. Let $k=\lceil \log n\rceil$. Writing $h_{2k}$ for the total weight of the homomorphic $2k$-cycles in $G'$ (rather than $G$), Lemma \ref{lem:weighted cycle count} and Corollary~\ref{cor:hom ineq} imply that
$$1\leq h_{2k}\leq \left(\frac{2k^2}{\delta(G')}\right)^k |V(G')|\leq \left(\frac{2k^2}{\delta(G')}\right)^k n\leq 3^{-k}n<1,$$
which is a contradiction.
\end{proof}

It remains to prove Theorem \ref{thm:almost rainbow cycle}. The proof uses suitable variants of Lemma \ref{lem:hom ineq one step} and Corollary \ref{cor:hom ineq}.
For these variants, we will need the following simple lemma.

\begin{lemma} \label{lem:cycles to circuits}
    Let $0<\eps<1/2$ and let $G$ be an edge-coloured graph in which for every $k$, every cycle of length $k$ has at most $(1-\eps)k$ different colours. Then for every $k$, every homomorphic $k$-cycle has at most $(1-\eps)k$ different colours.
\end{lemma}

\begin{proof}
We prove by induction on $k$ that every homomorphic $k$-cycle has at most $(1-\eps)k$ different colours. The statement is clear for $k=2$ since a homomorphic $2$-cycle has only one colour. Now let $k>2$ and let $C$ be a homomorphic cycle of length $k$ in $G$. If $C$ is a genuine cycle, then it follows from the assumptions that it has at most $(1-\eps)k$ different colours. Else, we can write $C$ as the concatenation of nontrivial homomorphic cycles $C_1$ and $C_2$. Writing $\ell$ and $k-\ell$ for the length of these homomorphic cycles, the induction hypothesis implies that $C_1$ has at most $(1-\eps)\ell$ different colours and $C_2$ has at most $(1-\eps)(k-\ell)$ different colours. It follows that $C$ has at most $(1-\eps)k$ different colours, completing the induction step.
\end{proof}

We can now state and prove the variant of Lemma \ref{lem:hom ineq one step}.

\begin{lemma} \label{lem:variant_hom ineq one step}
    Let $G$ be a properly edge-coloured graph with $\delta(G)>0$, let $0<\eps<1/2$ and let $k\geq 2$ be an integer. If for every $\ell$, $G$ has no cycle of length $\ell$ with more than $(1-\eps)\ell$ different colours, then $h_{2k}\leq \frac{k}{\eps\delta(G)} h_{2k-2}$.
\end{lemma}

\begin{proof}
By Lemma \ref{lem:cycles to circuits}, every homomorphic cycle of length $2k$ has at most $(1-\eps)2k$ different colours. Hence, any such cycle ``contributes" to $h_{2k}(i,j)$ for at least $\eps\cdot 2k$ pairs $(i,j)$ with $1\leq i<j\leq 2k$. Thus,
$$2\eps k h_{2k}\leq \sum_{1\leq i<j\leq 2k} h_{2k}(i,j).$$
By Lemmas \ref{lem:extremal pattern} and \ref{lem:step down}, this implies the desired inequality.
\end{proof}

The variant of Corollary \ref{cor:hom ineq} is as follows.

\begin{corollary} \label{cor:variant_hom ineq}
    Let $G$ be an $n$-vertex properly edge-coloured graph with $\delta(G)>0$, let $0<\eps<1/2$ and let $k\geq 2$ be an integer. If for every $\ell$, $G$ has no cycle of length $\ell$ with more than $(1-\eps)\ell$ different colours, then $h_{2k}\leq (\frac{k}{\eps\delta(G)})^k n$.
\end{corollary}

\begin{proof}
By repeated applications of Lemma \ref{lem:variant_hom ineq one step}, we obtain $h_{2k}\leq \frac{k!}{(\eps\delta(G))^{k-1}}h_2$. As we have seen in the proof of Corollary \ref{cor:hom ineq}, $h_2\leq \frac{n}{\delta(G)}$, which implies the result.
\end{proof}

\begin{proof}[Proof of Theorem \ref{thm:almost rainbow cycle}]
Let $n$ be sufficiently large, let $0<\eps<1/2$ and let $G$ be a properly edge-coloured $n$-vertex graph with at least $\frac{4}{\eps}n\log n$ edges. Then $G$ has a non-empty subgraph $G'$ with $\delta(G')\geq \frac{4}{\eps}\log n$.

Assume, for contradiction, that for every $\ell$, $G'$ has no cycle of length $\ell$ with more than $(1-\eps)\ell$ different colours. Let $k=\lceil \log n\rceil$. Writing $h_{2k}$ for the total weight of the homomorphic $2k$-cycles in $G'$ (rather than $G$), Lemma \ref{lem:weighted cycle count} and Corollary~\ref{cor:variant_hom ineq} imply that
$$1\leq h_{2k}\leq \left(\frac{k}{\eps\delta(G')}\right)^k |V(G')|\leq \left(\frac{k}{\eps\delta(G')}\right)^k n\leq 3^{-k}n<1,$$
which is a contradiction.
\end{proof}

\section{Concluding remarks} \label{sec:concluding remarks}

In this paper we proved the first power improvement over the dependent random choice bound for $\ex(n,Q_d)$. When $d$ is a power of two, such an improvement can be deduced from known results. Conlon and Lee \cite[Theorem 6.2]{CL21} showed that their Conjecture \ref{conj:conlonlee} holds for subdivisions of $d$-partite $d$-uniform hypergraphs. Here, for a hypergraph $\cH$, the subdivision of $\cH$ is the bipartite graph whose parts are $V(\cH)$ and $E(\cH)$ and in which $v\in V(\cH)$ is adjacent to $e\in E(\cH)$ if $v\in e$. It is not hard to see that $Q_d$ is a subdivision of a $d$-partite $d$-uniform hypergraph if and only if $d$ is a power of two. However, even for these values of $d$, the $\eps$ in $\ex(n,Q_d)=O(n^{2-1/d-\eps})$ coming from their result is smaller than exponential in $-d$, so much smaller than the one obtained in this paper. We remark that for a general value of $d$, the best known lower bound is $\ex(n,Q_d)=\Omega\Big(n^{2-\frac{2^d-2}{d2^{d-1}-1}}\Big)\geq \Omega(n^{2-2/d})$, coming from the probabilistic deletion method.

We have already mentioned that our method resembles that of another paper of Conlon and Lee \cite{CL17} in which they prove Sidorenko's conjecture for a certain class of graphs. The class of graphs their method applies to is similar to our ``reflective'' graphs (see our Definition \ref{def:reflective}): their graphs are also required to have many symmetric triples (see our Definition \ref{def:symmetric triples}) and it is needed that a certain reflection sequence, similar to the one in our Definition \ref{def:reflective}, on the set of edges eventually covers the entire edge set. However, the two sequences are slightly different (theirs runs on edges and ours runs on vertices) and it is not true that every graph for which their proof verifies Sidorenko's conjecture is reflective: e.g., the $2$-blowup of an even cycle of length at least six is not reflective, but their proof applies to it. Nevertheless, the similarity is close enough for it to make sense to look for further reflective graphs in their class of examples; indeed, this is how we chose the graphs from Definition \ref{def:set graph}. In this paper we have decided not to pursue this direction further.

It is worth mentioning that, building on Conlon and Lee's work \cite{CL17}, Coregliano \cite{Cor22} proved Sidorenko's conjecture for a family of graphs extending the family considered by Conlon and Lee, by studying a sequence of reflections on \emph{vertices}, similarly to our paper. However, the crucial condition in our Definition \ref{def:symmetric triples} that $R$ intersects both $A\cup F_{\phi}$ and $B\cup F_{\phi}$ means that the family of graphs for which his result applies to and the family of our reflective graphs are not identical.

\subsection*{Acknowledgements}

The first author is grateful to Cosmin Pohoata for bringing reference \cite{Liu21} to his attention and for useful discussions. We are also grateful to Noga Alon, Zach Hunter and the two referees for helpful comments.

\paragraph{Note added.} After this paper was written, we learnt that Kim, Lee, Liu and Tran \cite{KLLT22} independently showed that an $n$-vertex properly edge-coloured graph with at least $Cn(\log n)^2$ edges has a rainbow cycle (i.e., our Theorem \ref{thm:rainbow cycle}). For regular graphs their proof is similar to ours, but they use a different approach to deal with degree irregularities.

Ten months after posting our paper, Alon, Buci\'c, Sauermann, Zakharov and Zamir \cite{ABSZZ23} obtained an improved bound $O(n\log n\cdot \log \log n)$ for this problem.

\bibliographystyle{abbrv}
\bibliography{bibliography}

\end{document}